\documentclass[a4, 12pt]{amsart}
\usepackage{amssymb}
\usepackage{amstext}
\usepackage{amsmath}
\usepackage{amscd}
\usepackage{latexsym}
\usepackage{amsfonts}
\usepackage[all]{xy}	

\theoremstyle{plain}
\newtheorem{Theorem}{Theorem}[section]

\newtheorem{theorem}[Theorem]{Theorem}
\newtheorem{prop}[Theorem]{Proposition}

\newtheorem{lem}[Theorem]{Lemma}

\newtheorem{cor}[Theorem]{Corollary}

\newtheorem*{thm*}{Theorem}
\newtheorem*{cor*}{Corollary}

\newtheorem*{claim*}{Claim}

\theoremstyle{definition}

\newtheorem{dfn}[Theorem]{Definition}

\newtheorem{ex}[Theorem]{Example}

\theoremstyle{remark}

\newcommand{\Proof}{\begin{proof}}
\newcommand{\Qed}{\end{proof}}

\numberwithin{equation}{Theorem}

\def\x{\underline x}

\newcommand{\calD}{\mathcal{D}}

\newcommand{\calF}{\mathcal{F}}

\newcommand{\calM}{\mathcal{M}}

\newcommand{\fkm}{\mathfrak{m}}

\newcommand{\fkp}{\mathfrak{p}}
\newcommand{\fkq}{\mathfrak{q}}

\def\depth{\operatorname{depth}}

\def\Ann{\operatorname{Ann}}
\def\Ass{\operatorname{Ass}}
\def\Assh{\operatorname{Assh}}

\def\height{\operatorname{ht}}

\def\Spec{\operatorname{Spec}}

\def\mult{\operatorname{mult}}

\def\ardeg{\operatorname{arith-deg}}

\begin{document}

\setlength{\baselineskip}{17pt}
\title{Hilbert coefficients and sequentially Cohen-Macaulay modules }
\author{Nguyen Tu Cuong}
\address{Institute of Mathematics 18 Hoang Quoc Viet Road 10307 Hanoi Vietnam}
\email{ntcuong@math.ac.vn}

\author{Shiro Goto}
\address{Department of Mathematics, School of Science and Technology, Meiji University, 1-1-1 Higashimita, Tama-ku, Kawasaki 214-8571, Japan}
\email{goto@math.meiji.ac.jp}

\author{Hoang Le Truong}
\address{Institute of Mathematics 18 Hoang Quoc Viet Road 10307 Hanoi Vietnam}
\email{hltruong@math.ac.vn}
\thanks{ The first author is supported by NAFOSTED of Vietnam under grant number 101.01-2011.49. The third author is  partially supported by NAFOSTED of Vietnam and by RONPAKU fellowship.
\endgraf
\endgraf
}

\maketitle

\begin{abstract}
The purpose of this paper is to present a characterization of  sequentially
Cohen-Macaulay modules in terms of its Hilbert coefficients with respect to distinguished parameter ideals. The formulas involve arithmetic degrees. Among corollaries of the
main result we obtain a short proof of Vasconcelos Vanishing Conjecture for modules and an upper bound for the first Hilbert coefficient.
\end{abstract}

\section{Introduction}
Let $(R, \frak m)$ be a Noetherian local ring with the maximal ideal $\frak m$ and $I$ an $\frak m$-primary ideal of $R$. Let $M$ be a finitely generated $R$-module of dimension $d$.  It is well known that there exists a polynomial $p_I(n)$ of degree $d$ with rational coefficients, called the Hilbert-Samuel polynomial, such that  $\ell(M/I^{n+1}M)=p_I(n)$  for all large enough $n$. Then, there are integers $e_i(I,M)$ such that
$$p_I(n)=\sum\limits_{i=0}^d(-1)^ie_i(I,M)\binom{n+d-i}{d-i}.$$
These integers  $e_i(I,M)$   are called the Hilbert coefficients of $M$ with respect to $I$. In particular, the leading coefficient  $e_0(I, M)$ is  called the multiplicity of $M$ with respect to $I$ and $e_1(I,M)$ is called by Vasconcelos the Chern number of $I$  with respect to $M$. Although the theory of multiplicity has been rapidly developing for the last 50 years and proved to be a very important tool in algebraic geometry and commutative algebra, not so much is known about the Hilbert coefficients  $e_i(I,M)$ 
with $i>0$. At the conference in Yokohama 2008, W. V. Vasconcelos \cite{V2} posed the following conjecture:
\vskip0.2 cm
\noindent
{\bf The Vanishing Conjecture:}
Assume that $R$ is an unmixed, that is $\dim (\hat R/ P) = \dim R$ for all $P\in \Ass \hat R$, where $\hat R$ is the $\frak m$-adic completion of $R$. Then $R$ is  a Cohen-Macaulay local ring if and only if  $e_1(\frak q,R) = 0$ for some parameter ideal $\frak q$ of $R$.
\vskip0.2cm
\noindent
Recently, this conjecture has been settled by L. Ghezzi,  J.-Y. Hong. K. Ozeki, T. T. Phuong, W. V. Vasconcelos and the second author in \cite{GGHOPV}. Moreover, the second author showed in \cite{G} how one can use Hilbert coefficients of parameter ideals to study many classes of non-unmixed modules such as Buchsbaum modules, generalized Cohen-Macaulay modules, Vasconcelos modules.... The aim of our paper is to continue this research direction. Concretely, we will give  characterizations of a sequentially Cohen-Macaulay module in term of its  Hilbert coefficients with respect to  certain parameter ideals (Theorem \ref{coefficient}). Recall that  
sequentially Cohen-Macaulay module was introduced first by Stanley \cite{St} for graded case. In the local case, a module $M$ is said to be a sequentially Cohen-Macaulay module if  there exists a filtrations of submodules $M=D_0\supset D_1\supset\ldots\supset D_s$  such that  $\dim D_{i} >\dim D_{i+1}$ and $D_i/D_{i+1}$ are  Cohen-Macaulay for  all $i=0,1,\ldots,s-1$ (see \cite{Sc}, \cite{CN}). Then $M$ is a Cohen-Macaulay module if and only if $M$ is an unmixed sequentially Cohen-Macaulay module.  Therefore, as an immediate consequence of our main result, we get  again the answer to Vasconcelos' Conjecture for modules.  Furthermore, Theorem \ref{coefficient} let us to get several interesting properties of  the Chern numbers of parameter ideals on non-unmixed modules. Especially, we can prove a slight stronger than Theorem 3.5 of  M. Mandal, B. Singh and J. K. Verma in \cite{MSV} about the non-negativity of the Chern number  of   any parameter ideal with respect to arbitrary finitely generated module (Corollary \ref{negative}).\\
This paper is divided into 4 sections. In the next section we recall the notions of dimension filtration, good parameter ideals and distinguished parameter ideals following  \cite{Sc}, \cite{CN},  \cite {CC1},  \cite {CC2},  and prove some preliminary results on the dimension filtration. We discuss in Section 3 the relationship between  Hilbert coefficients and arithmetic degrees (see \cite{BM}, \cite{V}) of an $\frak m$-primary ideal. The last section  is devoted to prove the main result and its consequences.   

\vspace{3mm}
\section{The dimension filtration} 
Throughout this paper,  $(R,\fkm)$ is a Noetherian local ring and $M$ is a finitely generated $R$-module of dimension $d$. 
\begin{dfn} {(\cite{CC1},  \cite {CC2},  \cite{CN})}
A filtration $\calD:M=D_0\supset D_1\supset\ldots\supset D_s= H^0_\fkm(M)$ of submodules of $M$ is  said to be a {\it dimension filtration},  if $D_i$ is the largest submodule of $D_{i-1}$ with $\dim D_i<\dim D_{i-1}$ for all $i=1,\ldots,s$. 
A system of parameters $\x=x_1,\ldots,x_d$ of $M$ is called a {\it good system of parameters} of $M$, if  $N\cap (x_{\dim N+1},\ldots,x_d)M=0$ for all submodules $N$ of $M$ with $\dim N<d$. A parameter ideal $\frak q$ of $M$ is called a {\it good parameter ideal}, if there exists a good system of parameters  $\x=x_1,\ldots,x_d$ such that $\frak q = (\x)$.
 \end{dfn} 

Now let us briefly give some facts on the dimension filtration and good systems of parameters (see \cite{CC1},  \cite {CC2}, \cite{CN}).  Let $\Bbb N$ be the set of all positive integers. We denote by $$\Lambda(M)=\{r\in \Bbb N\mid \text{ there is a submodule } N \text{ of } M \text{ such that } \dim N=r\}.$$  Because of the Noetherian property of $M$,  the dimension filtration of $M$ and $\Lambda(M)$ exist uniquely. Therefore, throughout this paper we always denote  by 
$$\calD: M=D_0\supset D_1\supset\ldots\supset D_s=H^0_\fkm(M)$$
the dimension filtration of $M$  with $\dim D_i =d_i$, and $\calD_i=D_i/D_{i+1}$ for all $i=0,\ldots,s-1$. Then we can check that $$\Lambda(M)=\{d_i=\dim D_i\mid i= 0,\ldots,s-1\}.$$ 
In this case, we also say that the dimension filtration $\calD$ of $M$ has the length $s$. Moreover, let $\bigcap_{\fkp\in \Ass M}N(\fkp)=0$ be a reduced primary
decomposition of submodule $0$ of $M$, then $D_i=\bigcap_{\dim(R/\fkp)\geqslant d_{i-1}}N(\fkp)$. Especially, if we set
$\Assh(M)=\{\fkp\in \Ass(M)\mid\dim R/\fkp=\dim M\}$, then   the submodule $$D_1=\bigcap\limits_{\fkp\in \Assh(M)}N(\fkp)$$
 is called  the unmixed component of  $M$ and denoted by $U_M(0)$. It  should be mentioned that $U_M(0)$ is just the largest submodule of $M$ having the dimension strictly smaller than $d$. Moreover,  
  $H^0_\fkm(M)\subseteq U_M(0)$ and  $H^0_\fkm(M)=U_M(0)$ if $\Ass(M)\subseteq \Assh(M) \cup \{\fkm\}$. 
Put $N_i=\bigcap_{\dim(R/\fkp)\leqslant
d_i}N(\fkp)$. Therefore $D_i\cap N_i=0$ and $\dim (M/N_i)=d_i$. By the
Prime Avoidance Theorem,  there exists a system of parameters $\x=(x_1,\ldots,x_d)$ such that $x_{d_i+1},\ldots,x_d
\in\text {Ann }(M/N_i)$. It follows that $D_i\cap (x_{d_i+1},\ldots,x_d)M\subseteq N_i\cap D_i=0$ 
for all $i=1,\ldots,s$. Thus by the definition of the dimension filtration, $\x=x_1,\ldots,x_d$ is a good system of parameters of $M$, and therefore the set of good systems of parameters of $M$ is always non-empty. Let $\x=x_1,\ldots,x_d$ be a good system of parameters of $M$. It easy to see that   $x_1,\ldots,x_{d_i}$ is a good system of parameters of $D_i$, so is $x_1^{n_1},\ldots,x_d^{n_d}$ for any $d$-tuple of positive integers $n_1,\ldots , n_d$.  With notations as above we have

\begin{lem}\label{equi1}
Let $\calF:M=M_0\supset M_1\supset\ldots\supset M_t$ be a filtration of submodules of $M$. Then the following statements are equivalent:
\begin{enumerate}
\item[$(1)$] $\dim M_t\le 0$, $\dim M_{i+1}<\dim M_{i}$, and $\Ass(M_i/M_{i+1})\subseteq \Assh(M_i/M_{i+1})\cup \{\fkm\}$ for all $i=0, 1,\ldots,t-1$.
\item[$(2)$] $s=t$ and $D_i/M_i$ has a finite length for each $i=1,\ldots,s$.
\end{enumerate}
When this is the case, we have $\dim M_i=d_i$.
\end{lem}
\begin{proof}
$(2)\Rightarrow (1)$ is trivial from the definition of the dimension filtration.

$(1)\Rightarrow (2)$.  We show recursively   on $i$ that $D_i/M_i$ has a finite length for all $i\leqslant t$. In the case $i=1$, we have $M_1\subseteq D_1$ and so that $$\Ass(D_1/M_1)\subseteq \Ass(M/M_1)\subseteq \Assh(M/M_1)\cup\{\fkm\}=\Ass(M/D_1)\cup\{\fkm\}.$$ Thus $D_1/M_1=U_{M/M_1}(0)=H^0_\fkm(M/M_1)$. Hence $D_1/M_1$ has a finite length. Assume the result holds for $i$; we will prove it for $i+1$. Since $D_i/M_i$ has a finite length, we have $$\Ass(D_i/M_{i+1})\subseteq \Ass(M_i/M_{i+1})\cup\{\fkm\}=\Assh(M_i/M_{i+1})\cup\{\fkm\}.$$ Thus $\Assh(D_i/M_{i+1})=\Assh(M_i/M_{i+1})$ and so that $\Ass(D_i/M_{i+1})\subset \Assh(D_i/M_{i+1})\cup\{\fkm\}$.  Therefore, as similar in the case $i=1$, $D_{i+1}/M_{i+1}$ has a finite length. Hence $D_i/M_i$ has a finite length for all $i=1,\ldots,t$. The claim  $s=t$ follows from the definition of the dimension filtration and the fact that $\dim M_t \leqslant 0$.
\end{proof}

\begin{dfn} { (see \cite{Sc})}
Let $\calF:M=M_0\supset M_1\supset\ldots\supset M_t$ be a filtration of submodules of $M$. A system of parameters $\x=x_1,\ldots,x_d$ of $M$ is called a {\it distinguished system of parameters} of $M$ with respect to $\calF$, if  $(x_{\dim M_i+1},\ldots,x_d) \subseteq \Ann M_i$ for all positive integers $i$. A parameter ideal $\frak q$ of $M$ is called a {\it distinguished parameter ideal} of $M$ with respect to $\calF$, if there exists a distinguished system of parameters  $\x=x_1,\ldots,x_d$ of $M$ with respect to $\calF$ such that $\frak q = (\x)$. We simply say that $\fkq =(\x)$ is a distinguished  parameter ideal if  $\x$ is a distinguished system of parameters  with respect to the dimension filtration. 
 \end{dfn} 

Let $\calF:M=M_0\supset M_1\supset\ldots\supset M_t$ be a filtration of submodules of $M$. 
For each submodule $N$ of $M$, let 
$\calF/N:M/N =(M_0+N)/N \supset (M_1+N)/N\supset\ldots\supset (M_t+N)/N$ denote the filtration of submodules of $M/N$.
When $N=xM$ for some $x\in R$, we abbreviate $\calF/xM$ to $\calF_x$.

\begin{lem}
Let $\calF:M=M_0 \supset M_1 \supset\ldots\supset M_t$ be a filtration of submodules of $M$.  Then
the following statements hold true.
\begin{enumerate}
\item[$(1)$] A good system of parameters of $M$ is also a distinguished system of parameters of $M$ with respect to $\calF$. Thus, there always exists a distinguished system of parameters with respect to $\calF$.
\item[$(2)$] Let $N$ be a submodule of $M$. If $x_1,\ldots,x_d$ is a distinguished system of parameters of $M$ with respect to $\calF$ and $\dim N< \dim M$, then $x_1,\ldots,x_d$ is a distinguished system of parameters of $M/N$ with respect to $\calF/N$.
\end{enumerate}
\end{lem}
\begin{proof}
Straightforward.
\end{proof}

The following result of Y. Nakamura and the second author \cite{GN} is often used  in this section.
 
\begin{lem}{\cite{GN}}\label{finitely}
Let $R$ be a homomorphic image of a Cohen-Macaulay local
ring and assume that $\Ass(R)\subseteq \Assh(R)\cup \{\fkm\}$. Then
$$\calF=\{\fkp\in\Spec(R)\mid\height_R(\fkp)> 1=\depth(R_\fkp)\}$$
is a finite set.
\end{lem}

The next proposition shows the existence of a special superficial  element which is useful for many inductive proofs in the sequel.

\begin{prop}\label{exists}
Assume that $R$ is a homomorphic image of a Cohen-Macaulay
local ring. Let  $\fkq$ be a parameter ideal of $M$. Then there exists an element $x\in \frak q$ which is a superficial element of $D_i$ with respect to $\fkq$ such that $\Ass(\calD_i/x \calD_i)\subseteq \Assh(\calD_i/x \calD_i)\cup \{\fkm\}$, where $\calD_i=D_i/D_{i+1}$ for all $i=0,\ldots,s-1$ . Moreover, $x$ is also a regular element of $M/D_i$ for all $i=1,\ldots,s$.
\end{prop}
\begin{proof}

Set $I_i=\Ann(\calD_i)$, and $R_i=R/I_i$, then  $\Ass(R_i)=\Assh(R_i)$ and $\dim R/I_i> \dim R/I_{i+1}$ for all  $i=0,\ldots,s-1$. Moreover, we have $$\Ass(R_i)=\Ass(\calD_i)=\{\fkp\in\Spec(R)\mid\fkp\in\Ass(M)\text { and } \dim R/\fkp=\dim R/I_i=d_i\}.$$
Set  $$\calF_i=\{\fkp\in\Spec(R)\mid I_i\subset\fkp \text{ and } \height_{R_i}(\fkp/I_i)> 1=\depth((\calD_i)_\fkp)\}.$$
By Lemma \ref{finitely} and the fact $\Ass(\calD_i)=\Assh(\calD_i)$, we see that the set $$\{\fkp\in\Spec(R_i)\mid \height_{R_i}(\fkp)> 1=\depth((\calD_i)_\fkp)\}$$
is finite, and so that $\calF_i$ are a finite set for all $i=0,\ldots,s-1$. Put $\calF=\Ass(M)\cup\bigcup\limits_{i=1}^t\calF_i\setminus\{\fkm\}$. By the Prime Avoidance Theorem, we can choose $x\in\fkq-\fkm\fkq$ such that $x$ is a superficial element of $D_i$ with respect to $\fkq$ such that
$x\not\in\bigcup\limits_{\fkp\in\calF }\fkp$.  Since $x$ is a superficial element of $D_i$ and $\dim D_i>0$ for all $i=0,\ldots,s-1$, 
$\dim \calD_i/x \calD_i=\dim \calD_i-1$. Let $\fkp\in\Ass(\calD_i/x \calD_i)$ with $\fkp\not=\fkm $. Then we have $\depth(\calD_i/x\calD_i)_\fkp=0$. On the other hand, $\depth(D_i)_\fkp>0$ since $\fkp\not\in \Ass(\calD_i)\subseteq \Ass(M)$. Hence $\depth(\calD_i)_\fkp=1$. It implies that $\height_{R_i}(\fkp)=1$, since $\fkp\not\in\calF_i$. By the assumption $R_i$ is a catenary ring, therefore 
$$\dim R/\fkp= \dim R_i-\height_{R_i}(\fkp)=\dim R_i/xR_i=\dim \calD_i/x\calD_i.$$ Hence $\fkp\in \Assh (\calD_i/x \calD_i)$, and this completes the proof.
\end{proof}

\begin{lem}\label{filtration}
Let $R$, $M$ and $x$ be  as in the Proposition \ref{exists} and $\calD_{M/xM}:M/xM=D^\prime_0\supset D^\prime_1\supset\ldots\supset D^\prime_l$   the dimension filtration of $M/xM$. Then we have
 $$ l=
\begin{cases} s-1&\text{if $\dim D_{s-1}=1$,  }\\
s&\text{otherwise.}
 \end{cases}$$
 Moreover, $D_i^{\prime}/\overline D_i$ has a finite length, where $\overline D_i=(D_i+xM)/xM$, for all $i=0,\ldots,s-1$. 
\end{lem}
\begin{proof}
For  a submodule  $N$  of $M$ we set $\overline N=(N+xM)/xM$ a submodule of $M/xM$.
For all $i=1,\ldots,s$, since  $x$ is a regular element of $M/D_i$, we have $D_i\cap xM=xD_i$, and so that 
$$\begin{aligned}\calD_i/x\calD_i\cong D_i/(xD_i+D_{i+1})&\cong D_i/[D_i\cap(xM+D_{i+1})]
\\&\cong (D_i+xM)/(D_{i+1}+xM)=\overline D_i/\overline D_{i+1}.
\end{aligned}$$
Therefore,  the filtration of submodules of $M/xM$ $$\calD_x : M/xM=(D_0+xM)/xM\supset (D_1+xM)/xM\supset\ldots\supset (D_s+xM)/xM$$ satisfies the following conditions: for all $i=0,\ldots,s-1$ and $\dim D_{i+1}>0$, we have
$$\dim (D_i+xM)/xM>\dim (D_{i+1}+xM)/xM$$
and $$\Ass(M/xM)\setminus \{\fkm\}\subseteq\bigcup\limits_{i=0}^{s-1}\Ass(\overline D_{i}/\overline D_{i+1})\setminus \{\fkm\}.$$  
Thus, for all $\fkp\in \Ass(M/xM)\setminus\{\fkm\}$, there is an integer $i$ such that $\dim R/\fkp=\dim(\overline D_i/\overline D_{i+1})$.
Since $\calD_{M/xM}:M/xM=D^\prime_0\supset D^\prime_1\supset\ldots\supset D^\prime_l$ is the dimension filtration of $M/xM$, it follows that either  $l=s-1$ if $\dim D_{s-1}=1$, or $l=s$ otherwise. Moreover, we also obtain $\dim \overline D_i=\dim D_i^\prime$ and $\overline D_i\subseteq D_i^\prime$ for all $i=0,\ldots,s-1$.
Now we proceed by induction on $i$ to show that $D_i^\prime/\overline D_i$ has a finite length for each $i=1,\ldots,s-1$. In fact,  since $(D_1+xM)/xM=\overline D_1\subseteq D_1^{\prime}$ and $\dim D_1^{\prime}<\dim M/xM$,  $D_1^{\prime}/\overline D_1\subseteq U_{M/(D_1+xM)}(0)$. Moreover, by Lemma \ref{exists}, we obtain $\Ass(M/(D_1+xM))\subseteq\Assh(M/(D_1+xM))\bigcup\{\fkm\}$, and so that $U_{M/(D_1+xM)}(0)$ has a finite length. Hence $D_1^{\prime}/\overline D_1$ has a finite length.
Assume that the assertion holds for $i$, we will prove it for $i+1$. 
Since $$\Ass(\overline D_i/\overline D_{i+1})=\Ass((D_i+xM)/(D_{i+1}+xM))=\Ass(\calD_i/x\calD_i)\subseteq \Assh(\calD_i/x\calD_i)\bigcup\{\fkm\}$$ and $\Ass(D_i^{\prime}/\overline D_i)\subseteq \{\fkm\}$ by the inductive hypothesis, we get
 $$\Ass(D_i^{\prime}/\overline D_{i+1})\subseteq\Ass(\overline D_i/\overline D_{i+1})\bigcup\Ass(D_i^{\prime}/\overline D_i)\subseteq\Assh(\calD_i/x\calD_i)\bigcup\{\fkm\}.$$
 Therefore, it follows from the equality    $\dim\calD_i/x\calD_i=\dim \overline D_i=\dim D_i^\prime=\dim D_i^{\prime}/\overline D_{i+1}$  that $\Ass(D_i^{\prime}/\overline D_{i+1})\subseteq\Assh(\overline D_i/\overline D_{i+1})\bigcup\{\fkm\}=\Assh(D_i^{\prime}/\overline D_{i+1})\bigcup\{\fkm\}$. Thus $U_{D_i^{\prime}/\overline D_{i+1}}(0)$  has a finite length. Since $\overline D_{i+1}\subseteq D_{i+1}^{\prime}$ and $\dim(D_{i+1}^{\prime}/\overline D_{i+1})<\dim D_i^{\prime}/\overline D_{i+1}$, we have $D_{i+1}^{\prime}/\overline D_{i+1}\subseteq U_{D_i^{\prime}/\overline D_{i+1}}(0)$,  and therefore $D_{i+1}^{\prime}/\overline D_{i+1}$ has a finite length as required.
\end{proof}

\begin{cor}\label{df} Let $R$, $M$, and $x$ as in the Proposition \ref{exists}. Then
$$\Lambda(M/xM)=\{d_i -1 \mid  d_i=\dim D_i >1 , i= 0,\ldots,s-1\}.$$ 
\end{cor}
\begin{cor}\label{29}
Let $R$, $M$and $x$ be  as in the Proposition \ref{exists}. Let $\calF:M=M_0 \supset M_1 \supset\ldots\supset M_s$ be
 a filtration of submodules of $M$ such that  $D_i/M_i$ has a finite length for each $i=0,\ldots,s$. For a submodule $N$ of $M$ we set $\overline N =N/xN$. Let 
$\calD_{\overline M}: \overline M=D^\prime_0\supset D^\prime_1\supset\ldots\supset D^\prime_l$ be  the dimension filtration of $\overline M$. 
Assume that there exists  an integer
 $t_1\leqslant s$  such that $\dim \overline M_{t_1}\leqslant 0$. Then  the following conditions hold true. 
\begin{enumerate}
\item[$(1)$] $t_1 =l$ and $\dim\overline M_i<\dim \overline{M}_{i-1}$ for all $i=1,\ldots,t_1$.
\item[$(2)$] Either $t_1=s-1$ if $\dim D_{s-1}=1$, or $t_1=s$ otherwise. 
\item[$(3)$] For each $i=1,\ldots,s-1$, $D_i^{\prime}/\overline M_i$ has a finite length. 
\end{enumerate}
\end{cor}
\begin{proof} $(1)$ and $(2)$ are trivial by Lemma \ref{filtration}.\\
(3). For each $i=1,\ldots,s-1$, since $M_i$ is submodule of $D_i$ and $D_i/M_i$ has a finite length, $\overline M_i \subset \overline D_i$ and $\overline D_i/\overline M_i$ has a finite length. By Lemma \ref{filtration}, $D_i^{\prime}/\overline D_i$ has a finite length and  so has $D_i^{\prime}/\overline M_i$.
\end{proof}
\begin{lem}\label{30}
Let $R$, $M$, $\fkq$ and $x$ be  as in the Proposition \ref{exists}. Let $\calF:M=M_0 \supset M_1 \supset\ldots\supset M_s$ be
 a filtration of submodules of $M$ such that $D_i/M_i$ has a finite length for each $i=0,\ldots,s$. Assume that $\fkq$ is a distinguished parameter ideal of $M$ with respect to $\calF$. Then there exists a distinguished system of  parameters $x_1,\ldots,x_d$ of $M$ with respect to $\calF$ such that $x_1=x$ and $\fkq=(x_1,\ldots,x_d)$.  
\end{lem}
\begin{proof}
Since $\fkq$ is a distinguished parameter ideal of $M$ with respect to $\calF$, there exists a distinguished system of  parameters
 $y_1,\ldots,y_d$ of $M$ with respect to $\calF$ such that $\fkq=(y_1,\ldots,y_d)$ and $M_i\subseteq 0:_M y_j$ for all $j=d_i+1,\ldots,d$ and $i=1,\ldots,s$. In particular, we have $(y_{d_{s-1}+1},\ldots,y_d)\subseteq \Ann M_{s-1}$. Moreover, by Lemma \ref{equi1} we have $\dim D_{s-1}>0$ and $\Assh M_{s-1}=\Assh D_{s-1}=\{\fkp\in \Ass (M)\mid \dim R/\fkp=d_{s-1}\}$. Thus $(y_{d_{s-1}+1},\ldots,y_d) \subseteq \bigcap\limits_{\fkp\in \Ass (M), \dim R/\fkp=d_{s-1}}\fkp$. Since $\dim D_{s-1}>0$ and by the choice of $x$, the elements $x,y_{d_{s-1}+1},\ldots,y_d$  form a part of a minimal basis of $\fkq$. Thus $x,y_{d_{s-1}+1},\ldots,y_d$ is a part of a system of parameters of $M$. Therefore we can find $d_{s-1}$ elements $x_1=x,x_2,\ldots,x_{d_{s-1}}$ in $\fkq$ that such  $\fkq = (x_1,x_2,\ldots,x_{d_{s-1}},x_{d_{s-1}+1}=y_{d_{s-1}+1},\ldots,x_d=y_d)$ as required. 
\end{proof}


\section{Arithmetic degree and Hilbert Coefficients}
For prime ideal $\frak p$ of $R$, we define the length-multiplicity of $M$ at $\fkp$ as the length of $R_\fkp$-module $\Gamma_{\fkp R_\fkp}(M_\fkp)=H^0_{\fkp R_\fkp}(M_\fkp)$ and denote it by $\mult_M(\fkp)$. It is easy to see that $\mult_M(\fkp)\not=0$ if and only if $\fkp$ is an associated prime of $M$.
\begin{dfn}(\cite{BM},\cite{V},\cite{V1}) Let $I$ be an $\fkm$-primary ideal and $i$  a non-negative integer.  We define the {\it $i$-th arithmetic degree } of $M$ with respect to $I$ by
$$\ardeg_i(I,M)=\sum\limits_{\fkp\in\Ass(M),\:\dim R/\fkp=i}\mult_M(\fkp)e_0(I,R/\fkp).$$
 The {\it arithmetic degree} of $M$ with respect to $I$ is the integer
$$\begin{aligned}\ardeg(I,M)&=\sum\limits_{\fkp\in\Ass(M)}\mult_M(\fkp)e_0(I,R/\fkp)\\&=\sum\limits_{i=0}^{d}\ardeg_i(I,M).
\end{aligned}$$

\end{dfn}  
The following result gives a relationship between the multiplicity of submodules in the dimension filtration and the arithmetic degree.

\begin{prop}\label{muldeg} Let $(R,\fkm)$ be a local Noetherian ring, $I$ an $\fkm$-primary ideal and $\calD: M=D_0\supset D_1\supset \ldots\supset D_s=H^0_\fkm(M)$ the dimension filtration of $R$-module $M$. Then the following statements hold true.  
\begin{enumerate}
\item[$(1)$] $\ardeg_{0}(I,M)=\ell_R(H^0_\fkm(M))$.
\item[$(2)$] For $j=1,\ldots,d$, we have
 $$ \ardeg_{j}(I,M)=
\begin{cases} e_0(I,D_i)&\text{if $j=\dim D_i\in\Lambda(M)$, some i  }\\
0&\text{if $j\not\in \Lambda(M)$. }
 \end{cases}$$
 \end{enumerate}
 \end{prop}
\begin{proof} 
$(1)$ is trivial from the definition of the arithmetic degree.\\
$(2)$. By the associativity formula for multiplicities, we have $$e_0(I,D_i)=\sum\limits_{\fkp\in \Ass D_i,\:\dim R/\fkp=d_i}\ell((D_i)_\fkp)e_0(I,R/\fkp).$$
It follows from $\{\fkp\in \Ass(D_i)\mid\dim R/\fkp=d_i\}=\{\fkp\in\Ass(M)\mid\dim R/\fkp=d_i\}$ that  $H^0_{\fkp R_\fkp}(M_\fkp)\cong (D_i)_\fkp$
 for all $\fkp\in \Ass(M)$  with $\dim R/\fkp=d_i$. Thus
we get
$$\ell((D_i)_\fkp)=\ell(H^0_{\fkp R_\fkp}(M_\fkp))=\mult_M(\fkp)$$
for all $\fkp\in \Ass(M)$ and $\dim R/\fkp=d_i$. Hence 
$$e_0(I,D_i)=\ardeg_{d_i}(I,M),$$
for all $i=0,\ldots,s$. The rest of the proposition is trivial.
\end{proof}
For proving the main result in next section, we need two auxiliary lemmas as follows. It should be noticed that the statement (1) of Lemma \ref{nagata} below is also shown in \cite{MSV}, but the proof here is shorter.
\begin{lem}\label{nagata} Let $\frak q$ be a parameter ideal of $M$ with $\dim M = d$. Then  the following statements hold true.  
\begin{enumerate}
\item[$(1)$] If $d=1$, then $e_1(\frak q,M)=-\ell_R(H^0_\fkm(M))$.
\item[$(2)$] If $d\ge2$, then for every superficial element $x \in \frak q $ of $M$ it holds
$$ e_j(\frak q,M)=
 \begin{cases} e_j(\frak q,M/xM)&\text{if $0\le j\le d-2$,}\\
e_{d-1}(\frak q,M/xM)+(-1)^{d-1} \ell_R(0:_Mx)&\text{if $j=d-1$,}
 \end{cases}$$  
\end{enumerate}
\end{lem}
\begin{proof}
Let $d=1$ and $\frak q = (a)$. Choose the integer $n$  large enough such that $H^0_\fkm(M)=0:_Ma^n$ and $\ell(M/a^nM)=e_0((a),M)n-e_1((a),M)$. Then $$e_1((a),M)=-(\ell(M/a^nM)-e_0((a^n),M))=-\ell(0:_Ma^n)=-\ell(H^0_\fkm(M)).$$
The second statement was proved by M. Nagata \cite[22.6]{N}. 
\end{proof}
\begin{lem}\label{sum}
Let $N$ be a submodule of $M$ with $dim N=s<d$ and $I$ an $\fkm$-primary ideal of $R$. Then$$ e_j(I,M)=
\begin{cases} e_j(I,M/N)&\text{if $0\le j\le d-s-1$,}\\
e_{d-s}(I,M/N)+(-1)^{d-s} e_0(I,N)&\text{if $j=d-s$.}
 \end{cases}$$ 
\end{lem}		
\begin{proof}
From the  exact sequence
$$0\to N\to M\to M/N\to 0$$
we get the following exact sequence
$$0\to (N\cap I^nM)/I^nN\to N/I^nN\to M/I^nM\to M/I^nM+N\to0$$
 for each $n$. Thus 
 $$\ell(M/I^nM)=\ell(N/I^nN)+\ell(M/I^nM+N)-\ell((N\cap I^nM)/I^nN)$$
 for all $n$. Hence $\ell((N\cap I^nM)/I^nN)$ is a polynomial for  large enough $n$.
By the Artin-Rees lemma, there exists an integer $k$ such that $N\cap I^nM\subseteq I^{n-k}N$ for all $n\ge k$, and so that 
$$\ell((N\cap I^nM)/I^nN)\le\ell(I^{n-k}N/I^nN)\le\sum\limits_{i=n-k}^{n-1}\ell(I^iN/I^{i+1}N)$$
for all $n\ge k$. This gives that the degree of the polynomial  $\ell((N\cap I^nM)/I^nN)$ is strictly smaller than $\dim N$. Since $\dim N=s<d$, the conclusion follows by comparing coefficients of polynomials in the above equality.
\end{proof}

\section{Characterization of Sequentially Cohen-Macaulay modules}

The notion of sequentially Cohen-Macaulay module was introduced first by Stanley \cite{St} for graded case and  in  \cite{Sc}, \cite{CN}  for the local case.
\begin{dfn}  An $R$-module $M$ is called a {\it sequentially Cohen-Macaulay module} if there exists a filtration  $\mathcal{F:}M=M_0\supset M_1\supset\ldots\supset M_t$  of submodules of $M$ such that $\dim M_t\leqslant 0$, $\dim M_{i+1} < \dim M_{i}$ and $\calM_i=M_i/M_{i+1}$ are a Cohen-Macaulay module for all $i=0,\ldots,t-1$. 
\end{dfn}
It should be noticed here that if $M$ is a sequentially Cohen-Macaulay, the filtration $\mathcal F$ in the definition above is uniquely determined and it is just the dimension filtration $\calD: M=D_0\supset D_1\supset\ldots\supset D_s=H^0_\fkm(M)$ of $M$.  Therefore,  $M$ is always a sequentially Cohen-Macaulay module, if $\dim M=1$.
Now we give a characterization of sequentially Cohen-Macaulay modules having small dimension.
\begin{theorem}\label{coefficient=2} Let $M$ be a finitely generated $R$-module with $\dim M=2$.  Then the following statements are equivalent:
\begin{enumerate}
\item[$(1)$] $M$ is a sequentially Cohen-Macaulay $R$-module.
\item[$(2)$] For all  parameter ideals $\fkq$ of $M$
and $j=0, 1, 2$,
 we have $$e_j(\fkq,M)=(-1)^j\ardeg_{2-j}(\fkq,M).$$ 
\item[$(3)$] For all  parameter ideals $\fkq$ of $M$, we have 
$$e_1(\fkq,M)=-\ardeg_1(\fkq,M).$$
\item[$(4)$] For some  parameter ideal $\fkq$ of $M$, we have 
$$e_1(\fkq,M)=-\ardeg_1(\fkq,M).$$
\end{enumerate}
\end{theorem} 
\begin{proof}
$(1)\Rightarrow (2)$. The result follows  from the Propositions \ref{muldeg} and \ref{sum}. 

$(2)\Rightarrow (3)$ and $(3)\Rightarrow (4)$ are trivial. 

$(4)\Rightarrow (1)$. It suffices to show that $\overline M=M/D_1$ is a Cohen-Macaulay module. In fact, since $\dim D_1<\dim M=2$, then $\dim D_1=0$ or $1$.  If $\dim D_1=0$, then
$\ardeg_1(\fkq,M)=0$. Therefore we get by Lemma \ref{sum} and the hypothesis that $$e_1(\fkq,\overline M)=e_1(\fkq,M) =0.$$  If $\dim D_1 =1$,  it follows from Lemma \ref{sum} and Proposition \ref{muldeg} that  $$e_1(\fkq,\overline M)=e_1(\fkq,M) + e_0(\fkq,D_1)= e_1(\fkq,M) +\ardeg_1(\fkq,M) = 0.$$
  Thus in all cases we have $e_1(\fkq,\overline M) =0$.
 Choose now an element $x\in \frak q$ which is a superficial element of $\overline M$ with respect to $\fkq$. Then $x$ is an  $\overline M$-regular element, since $\Ass \overline M = \Assh \overline M$. It follows from the assumption 
$\dim \overline M=2$ and Lemma \ref{nagata} that 
$$0=e_1(\fkq,\overline M)=e_1(\fkq, \overline M/x\overline M)= -\ell(H^0_\fkm(\overline M/x\overline M)).$$
Thus $H^0_\fkm(\overline M/x\overline M)=0$. So depth$\overline M =2$ and  $\overline M$ is a Cohen-Macaulay module.
 \end{proof}

\begin{prop}\label{26}
Let $\calF:M=M_0 \supset M_1 \supset\ldots\supset M_s$ be
 a filtration of submodules of $M$ such that $D_i/M_i$ has a finite length for each $i=0,\ldots,s$, where $s$ is the length of the dimension filtration of $M$. Assume that $ M/D_j$ is a sequentially Cohen-Macaulay module for some $1\le j\le s$ and $x_1,\ldots,x_d$ is a distinguished system of parameters of $M$ with respect to $\calF$. Set $\fkq=(x_1,\ldots,x_d)$. Then the following statements hold true.
\begin{enumerate}
\item[$(1)$] For all $i=1,\ldots,j$, we have
$$(x_1,\ldots,x_d)^{n+1}M\cap D_i=(x_1,\ldots,x_{d_i})^{n+1}D_i,$$
for large enough $n$. 
\item[$(2)$] We have
$$\ell(M/\fkq^{n+1}M)=\sum\limits_{i=0}^{j-1}\binom{n+d_i}{ d_i}e_0(\fkq,D_i)+\ell(D_j/\fkq^{n+1}D_j),$$
for large enough $n$.
\end{enumerate}
\end{prop}
\begin{proof}
$(1)$. Let $j\leqslant s$ be a positive integer and $M/D_j$ a sequentially Cohen-Macaulay module. We prove statement (1) recursively on $i\leqslant j$.  Let $i=1$. Since $M/D_j$ is a sequentially Cohen Macaulay module,  $M/D_1$ is Cohen-Macaulay. Thus $(x_1,\ldots,x_d)^{k}M\cap D_1=(x_1,\ldots,x_d)^{k}D_1$ for  all $k$. Since $D_1/M_1$ is of finite length,  there exists a positive integer $n$ such that $(x_1,\ldots,x_d)^nD_1\subseteq M_1$. On the other hand, since $x_1,\ldots,x_d$ is a distinguished system of parameters of $M$ with respect to $\calF$, $(x_{d_1+1},\ldots,x_d)M_1= 0$. It follows for large enough $n$ that
$$\begin{aligned}
(x_1,\ldots,x_d)^{n+1}M\cap D_1&=(x_1,\ldots,x_d)^{n+1}D_1\\
&=(x_1,\ldots,x_{d_1})^{n+1}D_1+(x_{d_1+1},\ldots,x_d)(x_1,\ldots,x_d)^nD_1\\
&\subseteq (x_1,\ldots,x_{d_1})^{n+1}D_1+(x_{d_1+1},\ldots,x_d)M_1\\
&=(x_1,\ldots,x_{d_1})^{n+1}D_1.\\
\end{aligned}$$
 Therefore we get $(x_1,\ldots,x_d)^{n+1}M\cap D_1=(x_1,\ldots,x_{d_1})^{n+1}D_1$.
Assume now that the conclusion is true for  $i-1 <j$. Then we get 
$$\begin{aligned}
(x_1,\ldots,x_d)^{n+1}M\cap D_{i}&=((x_1,\ldots,x_d)^{n+1}M\cap D_{i-1})\cap D_i\\
&=(x_1,\ldots,x_{d_{i-1}})^{n+1}D_{i-1}\cap D_i.
\end{aligned}$$
Consider now the module $D_{i-1}$ with  two filtrations of submodules $\calF':D_{i-1} \supset M_{i} \supset\ldots\supset M_s$ and the dimension filtration
$\calD':D_{i-1} \supset D_i \supset\ldots\supset D_s$. It is easy to check that the module $D_{i-1}$ with these two filtrations of submodules  satisfies all of assumptions of the proposition. Thus, by applying our proof for the case $i=1$ with the notice that $x_1,\ldots,x_{d_{i-1}}$ is a distinguished system of parameters of $D_{i-1}$ with respect to $\calF'$ we have
$$\begin{aligned}
(x_1,\ldots,x_d)^{n+1}M\cap D_{i}=(x_1,\ldots,x_{d_{i-1}})^{n+1}D_{i-1}\cap D_i
=(x_1,\ldots,x_{d_i})^{n+1}D_i,
\end{aligned}$$
for large enough $n$, which finishes the proof of statement (1).\\
$(2)$ We argue by the induction on the length $s$ of the dimension filtration $\mathcal{D}$ of $M$. The case $s=0$ is obvious.  Assume that $s\geqslant j>0$. By virtue of the  statement (1) we get a short exact sequence
$$0\to D_1/\fkq^{n+1}D_1\to M/\fkq^{n+1}M\to M/\fkq^{n+1}M+D_1\to0$$
for large enough $n$. Therefore we have $$\ell(M/\fkq^{n+1}M)=\ell(D_1/(x_1,\ldots,x_{d_1})^{n+1}D_1)+\ell(\calD_0/\fkq^{n+1}\calD_0),$$
where $\calD_0=M/ D_1$.  
Since $x_1,\ldots,x_d$ is a distinguished system of parameters of $M$ with respect to $\calF$, $x_1,\ldots,x_{d_1}$ is a distinguished system of parameters of $D_1$ with respect to the filtration 
$D_1\supset M_2 \supset \ldots\supset M_s$. Notice that $D_1\supset D_2 \supset \ldots\supset D_s$ is the dimension filtration of $D_1$ and $D_k/M_k$ has a finite length for each $k=1,\ldots,s$.
Since $s\geqslant j>0$ and $M/D_j$ is a sequentially Cohen-Macaulay module, so is $D_1/D_j$. Because the dimension filtration of $D_1$ is of the length $s-1$, it follows from the inductive hypothesis that
$$\ell(D_1/(x_1,\ldots,x_{d_1})^{n+1}D_1)=\sum\limits_{i=1}^{j-1}\binom{n+d_i}{ d_i}e_0(\fkq,D_i)+\ell(D_j/\fkq^{n+1}D_j).$$
Since $\calD_0=M/D_1$ is Cohen-Macaulay of dimension $d=d_0$, we have $$\ell(\calD_0/\fkq^{n+1}\calD_0)=\binom{n+d}{d}e_0(\fkq,\calD_0)=\binom{n+d}{d}e_0(\fkq,D_0).$$ Hence$$\ell(M/\fkq^{n+1}M)=\sum\limits_{i=0}^{j-1}\binom{n+d_i}{ d_i}e_0(\fkq,D_i)+\ell(D_j/\fkq^{n+1}D_j),$$
for all large enough $n\geqslant0$ as required.
\end{proof}

\begin{prop}\label{31}
Let $R$ be a homomorphic image of a Cohen-Macaulay local ring and $M$ a finitely generated $R$ module of dimension $d=\dim M\geqslant 2$ . Let $\calF:M=M_0\supset M_1 \supset \ldots\supset M_s$ be
 a filtration of submodules of $M$ such that $D_i/M_i$ has a finite length for each $i=1,\ldots,s$. Assume that $\fkq$ is a distinguished parameter ideal of $M$ with respect to $\calF$ such that
for all $j\in \Lambda(M)$ we have 
$$e_{d-j+1}(\fkq,M)=(-1)^{d-j+1}\ardeg_{j-1}(\fkq,M).$$
Then $M$ is a sequentially Cohen-Macaulay module.  
\end{prop}
\begin{proof}
  Reminder that $\calD:M=D_0\supset D_1\supset\ldots\supset D_s= H^0_\fkm(M)$ is the dimension filtration of $M$ and $\Lambda(M)=\{d_i=\dim D_i\mid i= 1,\ldots,s-1\}$. For each $i=0,\dots, s-1$, we set
  $$ \hat{e}_{0}(\fkq,D_{i+1})=
\begin{cases} e_0(\fkq,D_{i+1})&\text{if $ d_{i+1}=d_i$ -1,  }\\
0&\text{otherwise.}
 \end{cases}$$
Then,   by virtue of Proposition \ref{muldeg}  the equality in the assumptions of our proposition can be rewritten as
 $$e_{d-d_i+1}(\fkq,M)=(-1)^{d-d_i+1}\hat{e}_0(\fkq,D_{i+1}) \hskip 2cm (*)$$ for all $i =0,\ldots,s-1$. 
We prove a statement  which   is  slight stronger than the proposition, but it is more convenient for the inductive process as follows: $M$ is sequentially Cohen-Macaulay if  the equations (*) hold true for all $d_i \in \Lambda(M)$ with $d_i >1$. We proceed by induction on $d$.
 The claim is  proved for the case 
$d=2$ by Theorem \ref{coefficient=2}. \\
 Suppose  that $d\geq 3$. Then  there exists  by Proposition \ref{exists} an element $x\in \frak q$ which is a superficial element of $D_i$ with respect to $\fkq$ such that $x$ is a regular element of $M/D_i$ for all $i=1,\dots, s$. For a submodule $N$ of M, we denote $\overline N=(N+xM)/xM$ the submodule of $M/xM$. Let $\calD_{M/xM}:M/xM=D^\prime_0\supset D^\prime_1\supset\ldots\supset  D^\prime_l$  be  the dimension filtration of $M/xM$ and $t\in\{0,\ldots,s\}$ an integer such that $\dim \overline M_t\leqslant 0$. By Corollaries \ref{df}, \ref{29} and Lemma \ref{30}, the filtration $\calF_x:M/xM=\overline M_0 \supset \overline M_1 \supset\ldots\supset \overline M_t$ of submodules of $M/xM$ satisfies the following conditions: 
\begin{enumerate}
\item[$(1)$] Either $t=l=s-1$ if $\dim D_{s-1}=1$, or $t=l=s$ otherwise. 
\item[$(2)$] For each $i=1,\ldots,s-1$, $D_i^{\prime}/\overline M_i$ has a finite length. 
\item[$(3)$] For each $i=1,\ldots,s-1$, $D_i^{\prime}/\overline D_i$ has  a finite length and $$\Lambda(M/xM)=\{d_i -1 \mid  d_i >1 , i= 1,\ldots,s-1\}.$$
\item[$(4)$] There exists a distinguished system of  parameters $x_1,\ldots,x_d$ of $M$ with respect to $\calF$ such that $x_1=x$ and $\fkq=(x_1,\ldots,x_d)$. Moreover, the system of  parameters $x_2,\ldots,x_d$ of $M/xM$ is a distinguished system of  parameters of $M/xM$ with respect to $\calF_x$.
\end{enumerate}
Now, we first show that the module $\overline M = M/xM$ satisfies all the assumptions of the proposition with the filtrations of submodules $\calD_{M/xM}$, $\calF_x$  and the distinguished parameter ideal $(x_2,\ldots,x_d)$ with respect to $\calF_x$.
Since  $x$ is a regular element of $M/D_i$ for all $i=1,\ldots,s$, 
we have $D_{i}\cap xM=xD_{i}$. Therefore 
$$\calD_{i}/x\calD_{i}\cong D_{i}/xD_{i}+D_{i+1}\cong D_{i}/[D_{i}\cap(xM+D_{i+1})]\cong (D_{i}+xM)/(D_{i+1}+xM).$$
It follows that if $d_{i}>1$ then  $ e_0(\fkq,\calD_{i}/x\calD_{i})=e_0(\fkq,\overline D_{i}/\overline D_{i+1})=e_0(\fkq, \overline D_i)= e_0(\fkq,D^\prime_{i})$ as  $D_i^{\prime}/\overline D_i$ is  of finite length, and so that 
$$e_0(\fkq,D_{i})=e_0(\fkq,\calD_{i}/x\calD_{i})= e_0(\fkq,D^\prime_{i}),$$
since  $x$ is also a non-zero divisor on $\calD_{i}$. 
Let $d_i -1 \in \Lambda (M/xM)$, and $d_i>2$. We consider the following two cases: If $d_{i+1} =d_i -1>1 $, and $\dim D_{i+1}^\prime= \dim D_{i+1}-1  =\dim D_i -2=\dim D_i^\prime-1$, therefore $\hat e_0(\fkq,D^\prime_{i+1})= e_0(\fkq,D^\prime_{i+1})$. Then, by applying Lemma \ref{nagata},  Proposition \ref{muldeg} 
  we get that
\begin{eqnarray*}
e_{(d-1)-(d_i -1) +1}(\fkq,M/xM)  &=& e_{d-d_i +1}(\fkq,M)\\
&=&(-1)^{d-d_i+1}e_0(\fkq,D_{i+1})\\
        &=& (-1)^{d-d_i+1}e_0(\frak q, D^\prime_{i+1})\\
        &=&(-1)^{(d-1)-(d_i -1) +1}\hat e_0(\frak q, D^\prime_{i+1}). 
\end{eqnarray*}
 If  $d_{i+1} \not=d_i -1$,  $\dim D_{i+1}^\prime \not=\dim D_i^\prime-1$, and so that $$\hat e_0(\fkq,D_{i+1})=\hat e_0(\frak q, D^\prime_{i+1})=0.$$
Thus
 $$e_{(d-1)-(d_i -1) +1}(\fkq,M/xM) = (-1)^{(d-1)-(d_i -1) +1}\hat e_0(\frak q, D^\prime_{i+1}) = 0. $$
This show that in both cases we obtain  
$$e_{(d-1)-(d_i -1) +1}(\fkq,M/xM) = (-1)^{(d-1)-(d_i -1) +1}\hat e_0(\frak q, D^\prime_{i+1}) $$ for all $d_i -1 \in \Lambda (M/xM)$ and $d_i-1>1$.  Therefore $M/xM$ is a sequentially Cohen-Macaulay module by the inductive hypothesis.\\   
Next, we prove by induction on $i$ that   
for all $i=0,\ldots,s-1$, if $d_i\ge 3$ then $D_{i+1}^\prime=\overline D_{i+1}$ and $D_{i}/D_{i+1}$ is a Cohen-Macaulay module. In fact,  let $i=0$. Since $\overline M/D^\prime_1$ is a Cohen-Macaulay module and $D_1^{\prime}/\overline D_1$ has a finite length, $H^i_\fkm(M/D_1+xM)=0$ for all $0<i<d-1$. Therefore, we derive from exact sequence
$$0\to M/D_1\overset{x}\to M/D_1\to M/D_1+xM\to 0$$  the  following exact sequence
 $$0\to H^0_\fkm(M/D_1+xM)\to H^1_\fkm(M/D_1)\overset{x}\to H^1_\fkm(M/D_1)\to 0.$$ 
 Thus  $H^1_\fkm(M/D_1)=0$, and so $D_1^{\prime}/\overline D_1=H^0_\fkm(M/D_1+xM)=0$. Hence $D_1^\prime=\overline D_1$. Moreover, since $x$ is  $\calD_0=M/D_1$-regular and  $\calD_0/x\calD_0\cong \overline M/\overline D_1 = \overline M/D^\prime _1$ a Cohen-Macaulay module, $\calD_0$  is a Cohen-Macaulay module.
Assume  now that $D_{j}^{\prime}=\overline D_{j}$ and $D_j/D_{j+1}$ are Cohen-Macaulay for all $j\leqslant i$ with $d_i\geq 3$. Then with the same argument as above, we can prove that  
$H^j_\fkm(\calD_i/x\calD_i)=0$ for all $0<j<d_i-1$. Therefore,   from the exact sequence
$$0\to \calD_i\overset{x}\to \calD_i\to \calD_i/x\calD_i\to 0$$ we obtain the following exact sequence
$$0\to H^0_\fkm(\calD_i/x\calD_i)\to H^1_\fkm(\calD_i)\overset{x}\to H^1_\fkm(\calD_i)\to0.$$ 
It follows  that $H^1_\fkm(\calD_i)=0$. Therefore $D_{i+1}^{\prime}=\overline D_{i+1}$. The Cohen-Macaulayness of $D_i/D_{i+1}$ follows from the fact  that $\calD_{i}/x\calD_{i}\cong \overline D_{i} /\overline D_{i+1} = D_{i}^\prime/D^\prime _{i+1}$ is a Cohen-Macaulay module.\\  
Denote by $N$ the largest submodule of $M$ such that $\dim N\le 2$. 
It should be mentioned that this submodule $N$ must be appeared in the dimension filtration of $M$, says $N= D_{k}$ for some $k\in \{s-2, s-1, s\}$. Then, from the proof above it is easy to see that
if $\dim N\leqslant 1$,  $M$ is sequentially Cohen-Macaulay. 
Assume that $\dim N =2$.  To prove  $M$ is sequentially Cohen-Macaulay in this case, it is remains to show that $N =D_k$ is sequentially Cohen-Macaulay.
By virtue of Lemma \ref{26}  we have for large enough $n$
$$\ell(M/\fkq^{n+1}M)=\sum\limits_{i=0}^{k-1}\binom{n+d_i}{ d_i}e_0(\fkq,D_i)+\ell(N/\fkq^{n+1}N).$$
 Therefore  by comparing coefficients of the equality above and by  hypotheses of the proposition we get 
$$\begin{aligned}
-e_1(\fkq,N)&=(-1)^{d-1}e_{d-1}(\fkq,M)\\
&=(-1)^{d-1}e_{d-2+1}(\fkq,M)\\
&=(-1)^{d-1}(-1)^{d-2+1}\ardeg_{2-1}(\fkq,M)\\
&=\ardeg_1(\fkq,M)=\ardeg_1(\fkq,N).
\end{aligned} 
$$
 Thus $N$ is a sequentially Cohen-Macaulay module by Theorem \ref{coefficient=2}, and the proof of the proposition is complete.
\end{proof}

We are now able to state our main result.
\begin{theorem}\label{coefficient} Assume that $R$ is a homomorphic image of a Cohen-Macaulay local ring.  Then the following statements are equivalent:
\begin{enumerate}
\item[$(1)$] $M$ is a sequentially Cohen-Macaulay $R$-module.
\item[$(2)$] For all  distinguished parameter ideals $\fkq$ of $M$
and $j=0,\ldots,d$,
 we have $$e_j(\fkq,M)=(-1)^j\ardeg_{d-j}(\fkq,M).$$ 
\item[$(3)$] For all  distinguished parameter ideals $\fkq$ of $M$ and $j\in \Lambda(M)$, we have 
$$e_{d-j+1}(\fkq,M)=(-1)^{d-j+1}\ardeg_{j-1}(\fkq,M).$$
\item[$(4)$] For some distinguished parameter ideal $\fkq$ of $M$  and for all $j\in \Lambda(M)$, we have 
$$e_{d-j+1}(\fkq,M)=(-1)^{d-j+1}\ardeg_{j-1}(\fkq,M).$$
\end{enumerate}
\end{theorem} 
\begin{proof}
$(1)\Rightarrow (2)$.  Since $M$ is a sequentially Cohen-Macaulay module, it follows from Proposition \ref{26} with $j=s$ that $$\ell(M/\fkq^{n+1}M)=\sum\limits_{i=0}^s\binom{n+d_i}{d_i} e_0(\fkq, D_i)$$
 for all distinguished parameter ideals $\fkq$ and large enough $n$. Therefore we get
 $$(-1)^{d-d_i}e_{d-d_i}(\fkq,M)=e_0(\fkq, D_i)$$ for all $i=0,\ldots,s$  and $ e_j(\fkq,M)=0$ for all  $j\not= d - d_i$. 
Therefore the conclusion follows  from the Proposition \ref{muldeg}. \\
$(2)\Rightarrow (3)$ and $(3)\Rightarrow (4)$ are trivial. \\
$(4)\Rightarrow (1)$ follows from the Proposition \ref{31}.
\end{proof}

The first consequence of Theorem \ref{coefficient} is  to give an affirmative answer for Vasconcelos' Conjecture announced in the introduction.  It is
noticed that recently  this conjecture has been settled in \cite{GGHOPV} and extended for modules in \cite [3.11]{MSV} provided $\dim R = \dim M$. 
\begin{cor}\label{1}
Suppose that $M$ is an unmixed $R$-module, that is $\dim (\hat R/ P) = \dim M$ for all $P\in \Ass _{\hat R}\hat M$, where $  \hat M$ is the $\frak m$-adic completion of $M$.  The  $M$ is a Cohen-Macaulay module if and only if $e_1(\fkq, M )\geq 0$ for  some parameter ideal $\fkq$ of $M$.
\end{cor}
\begin{proof}
Since  $M$ is unmixed,  we may assume without loss of generality that  $R$ is complete. Therefore $R$ is  a homomorphic image of a Cohen-Macaulay local ring and  $M=D_0 \supset D_1=0$ is the dimension filtration of $M$.  Thus  $\Lambda (M) = \{d\}$ and  $M$ is Cohen-Macaulay if it is sequentially Cohen-Macaulay. It follows  from Theorem \ref{coefficient} and the fact that every parameter ideal of $M$ is good  that  $M$ is Cohen-Macaulay if and only if there exists a parameter ideal $\frak q$ such that  $e_1(\frak q,M) =  -\ardeg_{d-1}(\fkq,M)$.  And the last condition is  equivalent to the condition $e_1(\fkq, M )\geq 0$ by Proposition \ref{muldeg}.
\end{proof}
The next corollary shows that the Chern number of a parameter ideal $\frak q$ is not only a non-positive integer but also bounded above by $ -\ardeg_{d-1}(\fkq,M)$. 
\begin{cor}\label{negative}
Let $M$ be a finitely generated $R$-module of dimension $d>0$.  Then $e_1(\fkq,M)\le -\ardeg_{d-1}(\fkq,M)$ for all parameter ideals $\fkq$ of $M$.
\end{cor}
\begin{proof}
Since the Hilbert coefficients and arithmetic degrees are unchanged by the $\frak m$-adic completion, we can assume that $R$ is complete. Then  by Lemma \ref{sum} we have
 $$e_1(\fkq,M) +\ardeg_{d-1}(\fkq,M)={e}_1(\fkq,M/U_M(0))  ,$$
 where $U_M(0)= D_{s-1}$ is the  unmixed part of $M$. If $e_1(\fkq,M)\ge -\ardeg_{d-1}(\fkq,M)$,  $e_1(\fkq,M/U_M(0))\ge 0$.  Then $M/U_M(0)$ is Cohen-Macaulay by Corollary \ref{1},  and so that $e_1(\fkq,M/U_M(0))= 0$. Hence $e_1(\fkq,M)\le -\ardeg_{d-1}(\fkq,M)$ for all parameter ideals $\fkq$ of $M$.
\end{proof}
The following immediate consequence of \ref {negative} is first proved in  \cite[Theorem 3.5]{MSV}.
\begin{cor}
Let $M$ be a finitely generated $R$-module of dimension $d>0$.  Then $e_1(\fkq,M)\le 0$ for all parameter ideals $\fkq$ of $M$.
\end{cor}
Below, we give some more corollaries in the cases that Chern numbers of parameter ideals have extremal values.
\begin{cor}\label{32}
Assume that $R$ is a homomorphic image of a Cohen-Macaulay local ring.  Then the following assertions are equivalent:
\begin{enumerate}
\item[$(1)$] $e_1(\fkq,M)= -\ardeg_{d-1}(\fkq,M)$ for all   parameter ideals $\fkq$ of $M$.
\item[$(2)$] $e_1(\fkq,M)= -\ardeg_{d-1}(\fkq,M)$ for some  parameter ideal $\fkq$ of $M$.
\item[$(3)$] $M/U_M(0)$ is  a Cohen-Macaulay module.
\end{enumerate}
\end{cor}

\begin{proof}
It follows immediately from  Theorem \ref{coefficient} and the fact that if $M$ is unmixed then every  parameter ideal of $M$ is  good.
\end{proof}
By virtue of Proposition \ref{muldeg} we see that  $\ardeg_{d-1}(\fkq,M)=0$ for some parameter ideal $\fkq$ of $M$ if and only if $\dim U_M(0)\le d-2$. Hence from this fact and Corollary \ref{32} we have
\begin{cor}\label{CM}
Assume that $R$ is a homomorphic image of a Cohen-Macaulay local ring.  The following assertions are equivalent:
\begin{enumerate}
\item[$(1)$] $e_1(\fkq,M)=0$ for all   parameter ideals $\fkq$ of $M$.
\item[$(2)$] $e_1(\fkq,M)= 0$ for some  parameter ideal $\fkq$ of $M$.
\item[$(3)$] $M/U_M(0)$ is Cohen-Macaulay module and $\dim U_M(0)\le d-2.$
\end{enumerate}
\end{cor}

In \cite {V2} Vasconcelos asked whether, for any two minimal reductions $J_1$, $J_2$ of an $\fkm$-primary ideal $I$,
$e_1(J_1,M) = e_1(J_2,M)$? 
As an application of  Corollary \ref{32} we get an answer to this question when $M/U_M(0)$ is a Cohen-Macaulay  module.
\begin{cor}
  Let  $I$ be an $\frak m$-primary ideal of $R$.  Assume that $R$ is  a homomorphic image of a Cohen-Macaulay local ring and $M/U_M(0)$ a Cohen-Macaulay $R$-module. Then  there exists a constant $c$ such that $e_1(J,M)= c $ for all minimal reductions  $J$ of $I$.
\end{cor}

\begin{proof}
Let $J$ be reduction of ideal $I$. Since  $M/U_M(0)$ is Cohen-Macaulay, 
we get by  Corollary \ref{32} and Proposition  \ref{muldeg} that
$$ e_1(J,M)=-\ardeg_{d-1}(J,M)=
\begin{cases} -e_0(J,U_M(0))&\text{if $\dim U_M(0)=d-1$,}\\
0&\text{if $\dim U_M(0)<d-1$.}
\end{cases}$$
 The conclusion follows  from a result of  D. G. Northcott and D. Rees  \cite{NR}, which says that  $e_0(J, U_M(0))= e_0(I,U_M(0))$  for all reduction ideals $J$ of $I$.
\end{proof}

It should be mentioned here that we do not need the assumptions that $R$ is a homomorphic image of a Cohen-Macaulay ring and the parameter ideal $\fkq$ is distinguished  in Theorem \ref{coefficient=2} for the case $\dim M\leqslant 2$. However,  these hypothese are essential in Theorem \ref {coefficient}. So we close this paper with the following two examples which show that the assumptions that $R$ is a homomorphic image of a Cohen-Macaulay ring and the parameter ideal $\fkq$ is distinguished in Theorem \ref {coefficient}, can not omit when $\dim M \geq 3$.

\begin{ex}
let  $k[[X,Y,Z,W]]$ be the formal power series ring over a field $k$. We consider the local ring  $S=k[[X,Y,Z,W]]/I,$
where $I=(X)\cap(Y,Z,W)$. Then $\dim S=3$ and  $\calD: S=D_0\supset (X)/I=D_1 \supset D_2= 0$ is the dimension filtration of $S$.  
By Lemma \ref{sum} we get that $e_1(Q,S)=e_1(Q, S/D_1)=0$ for every parameter ideal $Q$ of $S$.
On the other hand, there exists  by Nagata \cite{N} a Noetherian local integral domain $(R,\fkm)$ so that $\hat R= S$, where $\hat R$ is the $\fkm$-adic completion 
of $R$.  Let $\fkq$ be an arbitrary parameter  ideal of $R$. Since $R$ is a domain, $\fkq$ is distinguished. Moreover,  since $$e_1(\fkq,R)=e_1(\fkq S, S) = 0=-\ardeg_2(\fkq,R),$$  $R$ satisfies the condition (4) of Theorem \ref{coefficient}. But  $R$ is  not a sequentially Cohen-Macaulay domain, as it is not Cohen-Macaulay.  
\end{ex}

\begin{ex}
Let $R= k[[X,Y,Z,W]]$ be the formal power series ring over a field $k$. We look at the $R$-module 
$$M=(k[[X,Y,Z,W]]/(X,Y)\cap (Z,W))\bigoplus k[[X,Y,Z]]$$
Set $D_1=k[[X,Y,Z,W]]/(X,Y)\cap (Z,W)$. Then  $M\supset D_1\supset 0$ is  the dimension filtration
of $M$
and $\Lambda(M)=\{3;2\}$. Moreover, $D_1$ is a Buchsbaum  module, $\depth M=\depth D_1=1$ and so that
$M$ is not sequentially Cohen-Macaulay. We put $U=X-Z$, $V=Y-W$ and $Q=(U, V, X)$.
Since $M/D_1$ is Cohen-Macaulay, 
$e_i(Q,M/D_1)=0$ for all $i=1,2,3$. Therefore by Lemma \ref{sum} and Proposition \ref{muldeg} we have
$$e_1(Q,M)= -e_0(Q,D_1)=- \ardeg_2(Q,M), \text { and}$$
$$e_2(Q,M)= -e_1(Q,D_1)=0= \ardeg_1(Q,M).$$ 

\end{ex}



\begin{thebibliography}{GSa3}

\bibitem[BM]{BM} D. Bayer and D. Mumford, What can be computed on algebraic geometry?, {\it Computational Algebraic Geometry and Commutative algebra}, Proceedings. Cortona 1991(D. Eisenbud and L. Robbiano Eds), Cambridge University Press 1993 pp 1-48.
\bibitem[CC1]{CC1} N. T. Cuong and D. T. Cuong, \textit{dd-sequences and partial Euler-Poincare characteristics of Koszul complex}, J. Algebra Appl.  {\bf 6}, no. 2 (2007), 207-231.
\bibitem[CC2]{CC2} N. T. Cuong and D. T. Cuong, \textit{On sequentially Cohen-Macaulay modules}, Kodai Math. J.,  \textbf{30} (2007), 409-428.
\bibitem[CN]{CN} N. T. Cuong and L. T. Nhan, \textit{Pseudo Cohen-Macaulay and pseudo generalized Cohen-Macaulay modules}, J. Algebra, \textbf{267} (2003), 156-177.
\bibitem[CT]{CT} N. T. Cuong and H. L. Truong, {\it Parametric decomposition of powers of parameter ideals and sequentially Cohen-Macaulay modules}, Proc. Amer. Math. Soc. {\bf 137} (2009), no. 1, 19-26. 
\bibitem[GHV]{GHV} L. Ghezzi, J.-Y. Hong, W. V. Vasconcelos, {\it The signature of the Chern coefficients of
local rings}, Math. Res. Lett. {\bf 16} (2009), 279-289.
\bibitem[GGHOPV]{GGHOPV} L. Ghezzi, S. Goto, J.-Y. Hong. K. Ozeki, T. T. Phuong, and W. V. Vasconcelos,
{\it  The first Hilbert coefficients of parameter ideals},  J. London Math. Soc. (2) {\bf 81} (2010),679-695.
\bibitem[G]{G} S. Goto, {\it Hilbert coefficients of parameters},
  Proc. of the 5-th Japan-Vietnam Joint Seminar on Commutative Algebra, Hanoi (2010), 1-34.
\bibitem[GN]{GN} S. Goto and Y. Nakamura, {\it Multiplicity and Tight Closures of Parameters} J. Algebra  {\bf 244} (2001),  no. 1, 302-311.
\bibitem[MSV]{MSV} M. Mandal, B. Singh and J. K. Verma, {\it  On some conjectures about the Chern numbers of filtrations}, J. Algebra {\bf 325} (2011), 147-162. 
\bibitem[MV]{MV} C. Miyazaki and W. Vogel, {\it Towards a theory of arithmetic degrees},  Manuscripta Math.  {\bf 89}  (1996),  no. 4, 427--438. 
\bibitem[N]{N} M. Nagata, Local rings, {\it Interscience New York}, 1962.
\bibitem[NR]{NR} D. G. Northcott and D. Rees, {\it Reductions of ideals in local rings}, Proc. Camb, Philos. Soc. {\bf 50} (1954)145-158.
\bibitem[Sc]{Sc} P. Schenzel, \textit{On the dimension filtration and Cohen-Macaulay filtered modules}, In. Proc. of the Ferrara meeting in honour of Mario Fiorentini, University of Antwerp Wilrijk, Belgium, (1998), 245-264.
\bibitem[St]{St} R. P. Stanley, {Combinatorics and Commutative Algebra}, {\it Second edition, Birkh\H{a}user Boston}, 1996.
\bibitem[V]{V} W. V. Vasconcelos,  {\it The degrees of graded modules}, Lecture Notes in Summer School on Commutative Algebra, vol. 2, pp 141-196, Centre de Recerca Matematica, Bellaterra (Spain), 1996.
\bibitem[V1]{V1} W. V. Vasconcelos, {Computational Methods in Commutative algebra and Algebraic Geometry},  {\it Springer  Verlag, Berlin-Heidelderg-New York}, 1998
\bibitem[V2]{V2} W. V. Vasconcelos, {\it The Chern coefficients of local rings}, Michigan Math. J., {\bf 57} (2008), 725-743.
\end{thebibliography}
\end{document}